\documentclass[sigconf]{acmart}
\AtBeginDocument{%
  }

\copyrightyear{2026}
\acmYear{2026}
\setcopyright{usgovmixed}
\acmConference[KDD '26]{Proceedings of the 32nd ACM SIGKDD Conference on Knowledge Discovery and Data Mining V.2}{August 09--13, 2026}{Jeju Island, Republic of Korea}
\acmBooktitle{Proceedings of the 32nd ACM SIGKDD Conference on Knowledge Discovery and Data Mining V.2 (KDD '26), August 09--13, 2026, Jeju Island, Republic of Korea}
\acmDOI{10.1145/3770855.3818863}
\acmISBN{979-8-4007-2259-2/2026/08}
\usepackage{balance}
\usepackage{url}
\usepackage{multirow}

\begin{document}

\title{Graph Neural Multilevel Preconditioners for Iterative Solvers}

\author{Zechen Zhang}
\affiliation{%
    \institution{University of Minnesota}
    \city{Minneapolis}
    \state{MN}
    \country{USA}
}
\email{zhan5260@umn.edu}

\author{Rui Peng Li}
\affiliation{%
    \institution{CASA, Lawrence Livermore National Laboratory}
    \city{Livermore}
    \state{CA}
    \country{USA}
}
\email{li50@llnl.gov}

\author{Yousef Saad}
\affiliation{%
    \institution{University of Minnesota}
    \city{Minneapolis}
    \state{MN}
    \country{USA}
}
\email{saad@umn.edu}

\renewcommand{\shortauthors}{Zechen Zhang, Rui Peng Li, \& Yousef Saad}

\begin{abstract}

Solving large, sparse linear systems is a core task in scientific computing, and efficient iterative solvers rely critically on effective and robust preconditioning. While classical methods such as algebraic multigrid (AMG) are highly scalable, their robustness can degrade on indefinite or nonsymmetric systems where heuristics originally developed for elliptic PDEs are less reliable. Recently, Graph Neural Networks (GNNs) have emerged as data-driven preconditioners; yet, the practical impact of imposing an AMG-style hierarchy remains underexplored for general sparse matrices. In this work, we propose a Graph Neural Multilevel Preconditioner (GMP) that adopts an AMG hierarchy as a structural prior and learns smoothing, restriction, and interpolation operators in a unified framework. Our method targets general sparse systems and is instantiated as a drop-in preconditioner for standard Krylov solvers. On a benchmark of over 800 sparse matrices, we compare against classical AMG, single-level ILUT, and state-of-the-art GNN preconditioners, and characterize the regimes where multilevel graph neural preconditioning improves convergence or, conversely, introduces overhead relative to strong single-level baselines. These results highlight both the promise and the limitations of enforcing AMG-style multilevel structure in learned preconditioners for large-scale scientific simulations.

\end{abstract}

\ccsdesc[500]{Mathematics of computing~Numerical analysis}
\ccsdesc[500]{Mathematics of computing~Interpolation}
\ccsdesc[300]{Computing methodologies~Neural networks}

\keywords{Preconditioning; Graph Neural Networks; Multilevel Methods; Iterative Solvers}

\maketitle

\section{Introduction}
The computational cost of many scientific simulations is dominated by the performance of iterative solvers for sparse linear systems. To achieve fast convergence in complex applications, these solvers require sophisticated preconditioning tailored to the underlying matrix structure. A broad range of classical preconditioners has been developed, including incomplete factorizations~\cite{saad1994ilut}, sparse approximate inverses~\cite{chow1998approximate}, and algebraic multigrid~\cite{ruge1987algebraic}. Yet, designing preconditioners that remain robust across general sparse matrices, especially ill-conditioned, nonsymmetric, or indefinite systems, remains challenging. 

Recent work has explored data-driven preconditioning strategies that learn an approximate inverse directly from the matrix graph. In particular, \citet{chen2025graph} showed that single-level graph neural preconditioners based on graph convolution can be competitive in both stability and performance on ill-conditioned, nonsymmetric systems. Despite this promise, standard GCN-style architectures are fundamentally local: each layer aggregates information over one-hop neighborhoods, and enlarging the receptive field requires stacking many layers~\cite{kipf2017semisupervised}. Deep message passing is known to suffer from optimization difficulties~\cite{chen2018fastgcn} and over-smoothing~\cite{li18deep, Oono2020Graph}, where node embeddings become progressively indistinguishable and lose discriminative information. For sparse linear solvers, this locality can be especially limiting because the error components that stall Krylov iterations are often global (low-energy) modes that can be poorly eliminated by shallow, local operators~\cite{saad2003iterative}. Classical multilevel methods address exactly this spectral separation by constructing a hierarchy of scales so that inexpensive relaxation reduces high-frequency (local) error, while coarse-grid correction targets low-frequency (global) components~\cite{stuben2001review_amg}. However, this success comes with a major caveat: multilevel preconditioners are intrinsically coupled systems in which several operators---smoothers, grid transfer operators, and coarse solvers---must work coherently. In classical AMG, these operators are typically produced by hand-crafted rules that were developed primarily for elliptic PDEs. On sparse matrices that deviate from that regime, these heuristics can fail to produce effective methods, and the resulting preconditioner may stagnate or even destabilize Krylov iterations.

Refining heuristic components in classical AMG has motivated the development of data-driven multilevel methods. Early work in this direction largely followed a component-replacement strategy: specific AMG components were learned (e.g., smoothers~\cite{huang2022learning}, interpolation sparsity patterns~\cite{taghibakhshi2021optimization}, or their weights~\cite{luz2020learning}), while the overall Galerkin framework and V-cycle structure were retained. In these pipelines, coarse correction is typically still performed by a standard direct solve on the coarsest level; thus, learning primarily targets setup quality rather than replacing AMG end-to-end. Recently, multilevel learning has been explored in PDE settings, where hierarchy serves as an architectural prior for operator learning. For example, AMGNet~\cite{yang2022amgnet} uses AMG-inspired coarsening to enable message passing across multiple mesh scales for flow-field prediction, and M2NO~\cite{li2026m2no} leverages predefined multi-wavelet spaces within a multigrid-like design. A more recent work, ROBIN~\cite{wurth2025diffusionbased}, combines an AMGNet-style hierarchy with diffusion-based refinement for physical simulations. However, these approaches typically assume access to geometric meshes and are trained with supervised solution data (or trajectory ground truth), distinguishing them from the \emph{matrix-only} preconditioning setting targeted in this work.

We summarize our contributions as follows:
\begin{itemize}
    \item \textbf{Residual-conditioned multilevel preconditioning for matrix-only problems.}
    We integrate an AMG-generated hierarchy into a learnable, end-to-end pipeline that maps the current residual to a correction direction and is compatible with flexible Krylov solvers.

    \item \textbf{Bipartite cross-attention for learned grid transfer beyond SPD.}
    Motivated by the local reconstruction view of BAMG~\cite{brandt2011bootstrap}, we parameterize interpolation and restriction on AMG-induced bipartite graphs using cross-attention, employing distinct parameter sets for the two directions to support Petrov--Galerkin coarse operators.

    \item \textbf{Large-scale evaluation on challenging non-SPD and nonsymmetric systems.}
    We evaluate on 867 SuiteSparse~\cite{SuiteSparse} non-SPD matrices (up to 100K unknowns and 2M nonzeros) and compare against classical preconditioners and a strong single-level learned baseline, analyzing both convergence behavior and practical failure modes.
\end{itemize}

\section{Background}
\subsection{Algebraic Multigrid Methods}
Classical algebraic multigrid (AMG) achieves rapid convergence through two complementary components:
(i) relaxation, which efficiently attenuates error components that are oscillatory (``high energy'') with respect to $A$, and
(ii) coarse-grid correction, which removes the remaining algebraically smooth (``low energy'') error.
For symmetric positive definite (SPD) operators, the algebraic analogue of geometric smoothness is often characterized by a small Rayleigh quotient ${(e, e)_A}/{(e, e)}$.

The role of the coarse level is to represent the remaining low-energy components in a smaller subspace so they can be removed efficiently. Let $P\in\mathbb{R}^{n\times n_c}$ denote the interpolation operator, whose range $\mathrm{range}(P)$ defines the coarse space.
A standard requirement is that algebraically smooth errors be accurately approximated in $\mathrm{range}(P)$, which is formalized through an approximation property of the form
\begin{equation}
\label{eq:approx_prop}
\min_{e_c\in\mathbb{R}^{n_c}}\ \|e - P e_c\|_{M}^{2}\ \le\ C_P\, (e, e)_A,
\end{equation}
where $\|\cdot\|_M$ denotes a relaxation norm associated with the smoother, and the constant $C_P$ is mesh-independent. This inequality provides the two-grid link: the same quantity $(e,e)_A$ that tends to be small after relaxation controls how well the remaining error can be represented on the coarse level.

Given a restriction operator $R$, the coarse-grid correction is obtained by enforcing the Petrov--Galerkin condition
$e-Pe_c \perp \mathrm{range}(AR^\top)$, which leads to $RAPe_c = RAe$.
The operator $A_c = RAP$ is referred to as the Petrov--Galerkin coarse-grid operator, and the corresponding coarse-grid correction operator is
\begin{equation} \label{eq:cgc}
C = I - P(R A P)^{-1} R A.
\end{equation}
Together with pre- and post-relaxation with matrix $M$, the two-grid error propagation from one iteration to the next is given by
\begin{equation}
\label{eq:TG}
e'=(I-M^{-\top}A)C(I-M^{-1}A)e \equiv E_{TG}e.
\end{equation}
For an SPD matrix $A$, a typical choice is $R=P^\top$, in which case $C$ in Eq.~\eqref{eq:cgc} is an $A$-orthogonal projector (equivalently, the energy-minimizing best-approximation update) onto the complement of $\mathrm{range}(P)$; consequently, the correction is energy non-increasing:
$\|C e\|_A \le \|e\|_A$.

Equation~\eqref{eq:TG} shows that two-grid performance is determined by the triple $(M,P,R)$.
We next revisit how classical AMG constructs $P$ so that $\mathrm{range}(P)$ captures the smooth error remaining after relaxation. In practice, $P$ is constructed to be sparse by first selecting for each fine node $i$ a small coarse stencil $C_i$ (typically guided by strength-of-connection heuristics~\cite{stuben2001review_amg}), and subsequently determining the interpolation weights on this fixed pattern.

\subsubsection{Interpolation as Local Reconstruction}
\label{sec:BAMG}
After partitioning the vertex set $\mathcal{V}$ into two disjoint sets: fine nodes $\mathcal{V}_F$ and coarse nodes $\mathcal{V}_C$, the interpolation operator $P$ specifies how a coarse-grid vector is reconstructed on the fine nodes. For any fine node $i\in \mathcal{V}_F$, the interpolated value $(Pv)_i$ is defined as a weighted linear combination of the values at its coarse neighbors:
\begin{equation}
\label{eq:amg_interp_def}
(Pv)_i \;=\; \sum_{j \in C_i} w_{ij} \, v_j, \quad \text{for } i \in \mathcal{V}_F,
\end{equation}
where $C_i\subseteq \mathcal{V}_C$ denotes the interpolatory set of $i$, and $w_{ij}$ are the interpolation weights.
A central objective in computing $w_{ij}$ is to ensure that the local reconstruction in Eq.~\eqref{eq:amg_interp_def} is accurate for algebraically smooth modes, i.e., vectors satisfying $Av\approx 0$.

BAMG achieves this objective by utilizing a set of (smoothed) test vectors $\{v^{(m)}\}_{m=1}^{q}$,
to determine $w_{ij}$ through a weighted least-squares (L-S) minimization, where $r^{(m)} = A v^{(m)}$:
\begin{equation}
\label{eq:rbamg_ls}
\min_{\{w_{ij}\}}
\sum_{m=1}^q \omega_{m}
\left(
{v_i^{(m)} - \frac{r_i^{(m)}}{a_{ii}}} \;-\; \sum_{j\in C_i} w_{ij} v_j^{(m)}
\right)^2.
\end{equation} 

The term $r_i^{(m)}/a_{ii}$, introduced by the residual-based BAMG, corresponds to implicitly enforcing the $i$-th equation by one step of point-wise Jacobi relaxation.

BAMG can further refine interpolation adaptively by improving the test vectors using the current hierarchy and recomputing the local L-S problems. In practice, however, computing $P$ via Eq.~\eqref{eq:rbamg_ls} presents several challenges. First, a well-posed local L-S problem requires that the local test vectors (restricted to $C_i$) must form a basis for the $|C_i|$-dimensional subspace, which in turn requires $q \ge |C_i|$. This condition can be difficult to satisfy when $|C_i|$ varies widely across fine points. Second, for nonsymmetric or indefinite systems, the local algebraically smooth space may not be well-represented by a small local stencil $C_i$, potentially leading to nearly rank-deficient or ill-conditioned L-S systems and unstable weights. Finally, the quality of interpolation depends strongly on the choice of $C_i$ as well as the number and quality of the test vectors. Generating and adaptively updating these vectors can incur significant setup costs.

\subsubsection{Nonsymmetric and Indefinite Operators}
While constructing a stable interpolation operator $P$ is practically challenging even in the SPD setting, the difficulty is amplified for
nonsymmetric and indefinite systems.
These systems introduce a distinct left smooth space, so multilevel stability depends not only on $P$ but also on the construction of an appropriate restriction operator $R$.
In this regime, the clean SPD variational picture breaks down.
Since the algebraically smooth components associated with $A$ (right modes) and with $A^\top$ (left modes) can differ substantially, the choice $R=P^\top$ is no longer well-motivated.
The resulting coarse correction becomes an \emph{oblique} projection, which may be non-contractive in standard norms, making stability of multilevel correction a central challenge in nonsymmetric AMG~\cite{manteuffel2019convergence}.
Strong theoretical results on convergence of AMG methods for non-SPD matrices are difficult to establish.
Two-grid convergence for aggregation-based nonsymmetric AMG formulated in the $\sqrt{A^\top A}$-norm was introduced in~\cite{doi:10.1137/080727336}.
An optimal AMG theory for non-SPD matrices where the columns of $P$ and $R$ are right and left generalized eigenvectors can be found in~\cite{doi:10.1137/24M1679288}.
As a practical alternative to ideal restriction in the nonsymmetric setting, a local approximate ideal restriction strategy was proposed in $\ell$AIR~\cite{manteuffel2018nonsymmetric}.

These challenges motivate us to avoid explicitly enforcing \(R=P^\top\) 
or relying on problem-specific heuristics for transfer operators. Instead, we parameterize restriction and interpolation with separate, learnable models and optimize them jointly through a residual-based objective, ensuring that multilevel correction remains stable beyond the SPD setting. To lay the foundation for our learned operators, we first review the message passing neural network (MPNN) architecture used throughout this work.

\subsection{Message Passing Neural Networks (MPNNs)}
\label{sec:prelim_mpnn}

For an input graph $\mathcal{G}=(\mathcal{V},\mathcal{E})$, let $\mathbf{h}_i^k\in\mathbb{R}^{d_h}$ and $\boldsymbol{\varepsilon}_{ij}^k\in\mathbb{R}^{d_\varepsilon}$ denote 
the node and edge features at MPNN layer $k$. The update rules for a single message-passing layer are defined as:
\begin{align}
\boldsymbol{\varepsilon}_{ij}^{k+1} &= \phi_\varepsilon\!\left(\boldsymbol{\varepsilon}_{ij}^{k},\, \mathbf{h}_i^{k},\, \mathbf{h}_j^{k}\right),
\label{eq:mpnn_edge}\\
m_i^{k+1} &= \rho\!\left(\left\{\, \boldsymbol{\varepsilon}_{ij}^{k+1} : j\in\mathcal{N}(i)\,\right\}\right),
\label{eq:mpnn_agg}\\
\mathbf{h}_i^{k+1} &= \phi_v\!\left(\mathbf{h}_i^{k},\, m_i^{k+1}\right),
\label{eq:mpnn_node}
\end{align}
where $\phi_\varepsilon$ and $\phi_v$ are learnable multilayer perceptrons (MLPs) that update the edge and node features, respectively; $\rho$ is a permutation-invariant aggregator (e.g., sum, mean, or max); and $\mathcal{N}(i)$ denotes the set of in-neighbors of node $i$. Stacking $K$ such layers forms a depth-$K$ MPNN at each hierarchy level. In the following sections, we denote this $K$-layer update as
\begin{equation}
(\mathbf{h}^{K},\,\boldsymbol{\varepsilon}^{K})
=\mathrm{MPNN}\big(\mathbf{h}^{0},\, \boldsymbol{\varepsilon}^{0};\,A\big).
\end{equation}

\section{Residual-Conditioned Multilevel Graph Neural Preconditioner}

We now introduce the architecture of GMP, which mirrors classical AMG hierarchies. 

We define GMP as a learnable operator that maps an input residual $r$ to a correction $e$ for the residual equation $Ae=r$:
\begin{equation}
\label{eq:z}
e = \mathrm{GMP}_\theta(r; A, \mathcal{H}(A)),
\end{equation}
where $\mathcal{H}(A)$ denotes a 
multilevel hierarchy that serves as a structural prior. 

When used within an iterative solver, GMP acts as a nonlinear preconditioner, adaptively producing a correction direction $e$ based on the current system state $r$.
Fig.~\ref{fig:gmp_architecture} illustrates a simplified two-level GMP architecture. In the remainder of this section, we provide a detailed description of each component.

\begin{figure}[t!]
    \centering
    \includegraphics[width=0.46\textwidth]{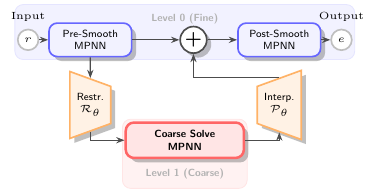}
    \caption{Two-level GMP V-cycle with learned operators. Given the fine-level residual $r$, a pre-smoothing MPNN generates an initial fine-grid correction. The residual is then mapped to the coarse level via a learned restriction $\mathcal{R}_{\theta}$. A coarse-solve MPNN computes a coarse correction, which is interpolated back to the fine level via a learned interpolation $\mathcal{P}_{\theta}$ and added to the initial correction (skip connection). Finally, a post-smoothing MPNN produces the final update $e$.}
    \Description{Illustration diagram of a two-level GMP V-cycle. On the fine level (Level 0), an input residual r enters a Pre-Smooth MPNN, and a Post-Smooth MPNN produces the output correction e, with a summation node combining the smoothed and interpolated corrections via a skip connection. A learned restriction operator R-theta maps information down to the coarse level (Level 1), where a Coarse Solve MPNN computes a coarse correction that is sent back up through a learned interpolation operator P-theta.}
    \label{fig:gmp_architecture}
\end{figure}

\subsection{Structural Prior}
\label{sec:structural_prior}
Given the input sparse matrix $A^{(0)} \in \mathbb{R}^{n_0 \times n_0}$, we construct a multilevel hierarchy that serves as the topological skeleton for our learned preconditioner. 

We utilize classical AMG routines to generate a sequence of progressively coarser operators $\mathcal{H} = \{A^{(\ell)}\}_{\ell=0}^{L}$ and the associated grid transfer operators $\{R^{(\ell)}, P^{(\ell)}\}$, where each coarse operator is defined via the Petrov--Galerkin condition $A^{(\ell+1)} = R^{(\ell)} A^{(\ell)} P^{(\ell)}$. We leverage this hierarchy to define both the graph topology and the message-passing features. The sparsity patterns of $A^{(\ell)}$ and the transfer operators determine the intra-level and inter-level connectivity, respectively. 

To account for the multilevel hierarchy and the multiple layers in MPNN, we employ a dual indexing notation $(\ell, k)$ in the following sections. The index $\ell$ denotes the AMG level, while $k\in\{0,\dots, K\}$ denotes the message-passing layer within level $\ell$. 

\subsection{Relaxation by Intra-Level Message Passing}
\label{sec:message_passing}
At each hierarchy level $\ell$, relaxation is performed on $\mathcal{G}_{A^{(\ell)}}$,
the graph induced by the sparsity pattern of 
the operator $A^{(\ell)}$. 
On $\mathcal{G}_{A^{(\ell)}}$, we treat the node feature $\mathbf{h}^{(\ell,k)}$ and edge attributes $\boldsymbol{\varepsilon}^{(\ell,k)}$ as the evolving algebraic state, which is updated by the MPNN described in Section~\ref{sec:prelim_mpnn}.

The node features are initialized as follows: at the finest level, we set $\mathbf{h}^{(0,0)} = r^{(0)}$, the fine-level residual; at coarser levels ($\ell > 0$), $\mathbf{h}^{(\ell,0)} = 0$, following the configuration in AMGNet~\cite{yang2022amgnet}. Edge features are initialized from the operator entries,
i.e.,
$\boldsymbol{\varepsilon}_{ij}^{(\ell,0)}=a_{ij}^{(\ell)}$. The relaxation is performed by $K$ message-passing layers:
\begin{equation}
(\mathbf{h}^{(\ell,K)},\,\boldsymbol{\varepsilon}^{(\ell,K)})
=\mathrm{MPNN}\!\left(\mathbf{h}^{(\ell,0)},\,\boldsymbol{\varepsilon}^{(\ell,0)};\,A^{(\ell)}\right),
\label{eq:level_relax_mpnn}
\end{equation}
where each layer follows the edge update, aggregation, and node update in
Eqs.~\eqref{eq:mpnn_edge}--\eqref{eq:mpnn_node}, and
the output $\mathbf{h}^{(\ell,K)}$ is the relaxed latent state.
This output serves as the input for the subsequent multilevel operator.
We do not decode $\mathbf{h}^{(\ell,\, K)}$ into a scalar correction, but instead maintain it in the high-dimensional feature space.

This design provides an alternative to an 
explicit residual update of the form $r 
\leftarrow r - A^{(\ell)} e$. 
In classical relaxation schemes, such a subtraction arises because 
the residual is
explicitly maintained as a separate 
vector throughout the iterations. 
In contrast, in our formulation, the iteration state is the node feature $\mathbf{h}$ itself, and the MPNN update has direct access to the local action of $A^{(\ell)}$
on the current state through message passing. Since the edge features $\boldsymbol{\varepsilon}_{ij}$ are initialized from matrix entries $a_{ij}^{(\ell)}$, the aggregation in Eq.~\eqref{eq:mpnn_agg} can represent neighborhood summations analogous to
sparse matrix-vector products. 
More specifically, 
the aggregated message at node $i$ can be
written as
\begin{equation}
m_i^{(\ell,\,k+1)} \;=\; \sum_{j\in\mathcal{N}_\ell(i)} a_{ij}^{(\ell)}\,\phi_{\varepsilon}\!\left(\mathbf{h}_j^{(\ell,\,k)}\right),
\label{eq:mpnn_Ah_signal}
\end{equation}
where $\phi_\varepsilon$ is a learnable map
induced by the edge attribute update in Eq.~\eqref{eq:mpnn_edge}.
The node update in Eq.~\eqref{eq:mpnn_node} then combines $\mathbf{h}_i^{(\ell,k)}$ and the $A^{(\ell)}$-conditioned signal $m_i^{(\ell,k+1)}$ to produce $\mathbf{h}_i^{(\ell,k+1)}$. 
In this way, residual-style corrections 
can be implemented implicitly within the
feature update. 
Consequently, the operation 
``$r - A^{(\ell)}(\cdot)$''  
is absorbed into the node update map $\phi_v$,
rather than being performed as a separate sparse matrix-vector product 
followed by a subtraction.

\subsection{Grid Transfer via Bipartite Cross-Attention}
\label{sec:bipartite_att}

In this section, we discuss the construction of the learnable grid transfer operators:
the interpolation operator 
$\mathcal{P}_\theta:\mathcal{V}_C\to \mathcal{V}_F$ and the
restriction operator
$\mathcal{R}_\theta:\mathcal{V}_F\to \mathcal{V}_C$. 
We further establish the connection between the BAMG  scheme introduced in Sec.~\ref{sec:BAMG} and graph attention mechanism on bipartite graphs.

\subsubsection{Graph Cross-Attention as Dynamic Interpolation}
Interpreting AMG interpolation as a form of neighborhood aggregation suggests that the weights encode the relative importance of neighboring coarse nodes when reconstructing smooth modes at fine points. 
While BAMG computes these weights through explicit local L-S solves, the underlying objective, i.e., accurate local reconstruction of smooth error using neighboring information, can be expressed more generally. Graph attention mechanisms~\cite{gat, shi2020masked} offer a flexible framework for parameterizing such local aggregation rules. In our framework, we interpret the attention mechanism not merely as feature aggregation, but as a \textit{learnable, dynamic solver} for the interpolation weights.

To formalize this, let $\mathcal{G}_{P}^{\mathcal{B}} = (\mathcal{V}, \mathcal{E})$ 
denote the graph induced by 
the sparsity pattern of the interpolation operator $P \in \mathbb{R}^{n \times n_c}$. 
In this context, the graph is naturally \emph{bipartite}, partitioning the vertex set $\mathcal{V}$ into two disjoint sets: $\mathcal{V}_C$ (coarse nodes) and $\mathcal{V}_F$ (fine nodes). The edge set $\mathcal{E}$ is restricted to pairs $(i,j)$ where $i\in\mathcal{V}_F$ and $j\in\mathcal{V}_C$, such that the neighborhood $\mathcal{N}(i)$ of a fine node $i$ corresponds exactly to its immediate interpolatory set $C_i$.

To ensure end-to-end training across all hierarchy components, we use the output of the intra-level message passing at level $\ell$ to initialize the $\ell\to\ell-1$ bipartite attention operator $\mathcal{P}_{\theta}$. Specifically, we set the node and edge features as follows:
\begin{align}
\mathbf{h}_{i}^{(\ell-1,\,0)} &= \mathbf{h}_{i}^{(\ell, K)}, \quad i \in \mathcal{V}_C \\
\mathbf{h}_{i}^{(\ell-1,\,0)} &= (P^{(\ell-1)}\;\mathbf{h}^{(\ell, K)})_i, \quad i \in \mathcal{V}_F\\
\boldsymbol{\varepsilon}_{ij}^{(\ell-1,\,K)} &= P_{ij}^{(\ell-1)},
\end{align}
where $\mathbf{h}_{i}^{(\ell, K)}$ is the output of the relaxation MPNN in Eq.~\eqref{eq:level_relax_mpnn}, $P^{(\ell-1)}$
is the interpolation operator generated by the classical AMG algorithm described in Sec.~\ref{sec:structural_prior},
and the edge features are initialized by the entries in $P$. For the remainder of this subsection, we omit the level superscripts on $\mathbf{h}$ and $\boldsymbol{\varepsilon}$ to simplify notation.

These embeddings, having aggregated local information via the operator $A^{(\ell)}$, effectively serve as learned surrogates for the algebraically smooth error modes. We denote by $W_{\text{self}}, W_{\text{val}}, W_{\text{edge}}$ 
the neural weight matrices used to transform the fine node's own features, project the features of its neighboring coarse nodes,
and encode edge attributes, respectively. A bipartite graph attention layer models the interpolation operator
$P$ by updating the fine node features $\mathbf{h}_i$ via an aggregation of coarse-grid corrections 
$\mathbf{h}_j$:
\begin{equation}
\label{eq:attn_generic_update}
\mathbf{h}_i \;=\; W_{\text{self}} \mathbf{h}_i \;+\; \sum_{j\in C_i} \alpha_{ij}\, \bigl(W_{\text{val}} \mathbf{h}_j + W_{\text{edge}} \boldsymbol{\varepsilon}_{ij}\bigr), \quad \text{for } i \in \mathcal{V}_F.
\end{equation}

Comparing Eq.~\eqref{eq:attn_generic_update} with classical interpolation in Eq.~\eqref{eq:amg_interp_def}, the attention coefficients $\alpha_{ij}$ play precisely the role of the interpolation weights $w_{ij}$. However, graph attention predicts these weights dynamically using a parameterized scoring function. 

\begin{figure}[t]
    \centering
    \includegraphics[width=0.36\textwidth]{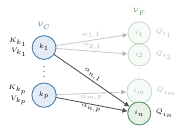}
    \caption{Interpolation via Coarse-to-Fine Attention. Active nodes (solid) illustrate the aggregation process for a fine query node $i$, gathering information from its coarse neighbor.}
    \Description{Diagram of coarse-to-fine attention on a bipartite graph between fine nodes and coarse nodes. A highlighted fine query node i aggregates information from its connected coarse neighbor node, illustrating how interpolation weights are predicted by a dot-product attention mechanism in which queries come from fine nodes and keys and values come from coarse nodes and their connecting edges.}
    \label{fig:bipartite_mechanism}
\end{figure}

Specifically, we employ a dot-product attention mechanism similar to graph transformers~\cite{shi2020masked}, where the roles of Query, Key, and Value are strictly partitioned according to the grid hierarchy, as illustrated in Fig.~\ref{fig:bipartite_mechanism}:
\begin{itemize}
    \item \textbf{Queries ($\mathbf{q}_i$):} Derived from the \textit{target fine nodes} ($i \in \mathcal{V}_F$), representing the local error state requiring correction.
    \item \textbf{Keys ($\mathbf{k}_{ij}$) and Values ($\mathbf{v}_{ij}$):} Derived from the \textit{source coarse nodes} ($j \in \mathcal{V}_C$) and the connecting edges, providing the coarse-grid information.
\end{itemize}
We employ the projection matrices $W_Q, W_K \in \mathbb{R}^{d \times d}$ and $W_E \in \mathbb{R}^{d \times d_\varepsilon}$ to map these features into the attention space, and compute the attention score as:
\begin{equation}
\label{eq:trans_alpha}
\alpha_{ij}
\;=\;
\text{softmax}_{j\in C_i}
\left(
\frac{(W_Q \mathbf{h}_i)^\top (W_K \mathbf{h}_j + W_E \boldsymbol{\varepsilon}_{ij})}{\sqrt{d}}
\right).
\end{equation}
In this formulation, the attention score measures the compatibility between the fine-grid error mode (Query) and  coarse-grid information (Key), determining an adaptive interpolation weight $\alpha_{ij}$ that updates the edge attribute $\boldsymbol{\varepsilon}_{ij}$. The weights derived from the attention mechanism inherently guarantee that the resulting interpolation operator is stable in the $\ell_\infty$-norm and preserves constant vectors; see Appendix~\ref{sec:theorem_a1} for details.

\subsubsection{Graph Cross-Attention as Restriction}
The restriction operator $\mathcal{R}_\theta$, which maps from the fine grid to the coarse grid, is realized by reversing the direction of the bipartite graph attention mechanism. In this setting, the roles are inverted when computing the node feature update: the coarse nodes $\mathcal{V}_C$ act as Queries, while the fine nodes $\mathcal{V}_F$ serve as Keys and Values. 

By employing a distinct set of learnable parameters for this reverse direction, this formulation becomes compatible with nonsymmetric 
Petrov--Galerkin coarse-grid operators constructed
via $A_c = \mathcal{R}_\theta A \mathcal{P}_\theta$ where $\mathcal{R}_\theta \neq \mathcal{P}_\theta^\top$. This added flexibility is 
particularly
crucial for nonsymmetric systems, allowing the model to learn restriction weights that are specifically optimized for aggregating residuals, distinct from the interpolation weights used for correcting errors.

\subsubsection{Connections between BAMG and Graph Cross-Attention} 
We conclude by examining the structural correspondence between AMG and our proposed learnable transfer operator. 
In BAMG, the ``data'' used to determine interpolation weights are the smooth test vectors $v^{(m)}$, and the weights are derived by solving local L-S problems. The graph attention framework mirrors this logic but generalizes the representation: the learned node embeddings $\mathbf{h}_i$ serve as high-dimensional surrogates for the algebraically smooth error modes. Generated through stacked message-passing layers, these embeddings aggregate features from an expanded receptive field, capturing structural dependencies from multi-hop neighborhoods rather than relying solely on direct adjacencies. From this perspective, the dot product $(W_Q \mathbf{h}_i)^{\top} (W_K \mathbf{h}_j)$ in Eq.~\eqref{eq:trans_alpha} serves as a learned measure of alignment between nodes $i$ and $j$, analogous to the correlation terms ($v_i^{(m)},v_j^{(m)}$) that appear in the normal equations of the local L-S systems of BAMG. 
By replacing the explicit solution of these L-S systems with a differentiable attention mechanism, the proposed approach enables the adaptive generation of interpolation weights. This potentially improves robustness in regimes where the local L-S problems become ill-conditioned, such as for nonsymmetric or indefinite matrices.

\subsection{Coarse-Level Correction by Message Passing}

We next describe a simplified two-level architecture by specifying the node feature updates at every stage and discussing the coarse solve. A multilevel architecture and the training scheme can be extended naturally by applying the two-level approach
recursively.

Let \(\mathbf{h}_{\mathrm{pre}}^{(0,K)}\) denote the output of the pre-smoothing MPNN applied to 
the fine-level graph associated with \(A^{(0)}\); see Eq.~\eqref{eq:level_relax_mpnn}.
The coarse-level node features are initialized by restricting the fine-level latent state using 
the learned restriction operator:
\begin{equation}
\mathbf{h}_c^{(1,0)} \;=\; \mathcal{R}_{\theta}\!\left(\mathbf{h}_\mathrm{pre}^{(0,K)}\right),
\label{eq:restrict_latent}
\end{equation}
where \(\mathcal{R}_{\theta}\) is parameterized via 
the bipartite cross-attention mechanism described in Section~\ref{sec:bipartite_att}.
For the coarse-solve MPNN, the coarse-level 
edge features are initialized from the 
entries of the
coarse-grid operator:
\begin{equation}
\boldsymbol{\varepsilon}_{ij}^{(1,0)} \;=\; a^{(1)}_{ij}, \qquad (i,j)\in\mathcal{E}^{(1)},
\label{eq:coarse_edge_init}
\end{equation}
where $a_{ij}^{(1)}$ is an entry of the coarse-grid operator $A^{(1)}$, the structural prior constructed following Sec.~\ref{sec:structural_prior}.

On the graph induced by the structure of
the coarse-grid operator
\(A^{(1)}\), we apply a dedicated coarse-solve MPNN with a parameter set distinct from that used in the fine-level smoother MPNNs:
\begin{equation}
\big(\mathbf{h}_c^{(1,K)},\,\boldsymbol{\varepsilon}_c^{(1,K)}\big)
\;=\;
\mathrm{MPNN}_{c}\!\left(\mathbf{h}_c^{(1,0)},\,\boldsymbol{\varepsilon}^{(1,0)};\,A^{(1)}\right).
\label{eq:coarse_mpnn}
\end{equation}

The resulting coarse-level correction is interpolated back to the fine level using the learned operator \(\mathcal{P}_{\theta}\) and combined additively with 
the pre-smoothed fine-level state:
\begin{equation}
{\mathbf{h'}}
\;=\;
\mathbf{h}^{(0,K)}_{\mathrm{pre}}
\;+\;
\mathcal{P}_{\theta}\!\left(\mathbf{h}_c^{(1,K)}\right).
\label{eq:additive_cgc}
\end{equation}
The post-smoothing MPNN takes \({\mathbf{h'}}\) as its input state on the fine level and produces the final correction \(e\); see Fig.~\ref{fig:gmp_architecture}.

\section{Experiments}
\label{sec:experiments}

We evaluate the proposed GMP as a drop-in preconditioner for Krylov subspace methods on 
sparse linear systems arising from a variety of applications.
Our experiments are designed to address the following questions:

\noindent\textbf{Q1}: Is a multilevel architecture 
beneficial for machine-learning-based preconditioning?
    
\noindent\textbf{Q2}: GNN-based preconditioners are known for their robustness~\cite{chen2025graph}; does
    introducing a 
    multilevel structure compromise this property?
    
\noindent\textbf{Q3}: How do learned preconditioners compare with classical methods in terms of both convergence and timing behavior?

\noindent\textbf{Q4}: How does GMP perform in challenging cases where classical AMG methods struggle to provide effective preconditioners?

\subsection{Experiment settings}
\subsubsection{Baseline Setup}

To evaluate our method across diverse applications,  
we select 867 square, real-valued, and non-SPD matrices in the size range of $1$K to $100$K with fewer than $2$M nonzeros from the SuiteSparse Matrix Collection~\cite{SuiteSparse}.
The chosen problems span over 50 application areas, including PDEs, economics, and graph problems.\footnote{Code available at \url{https://github.com/zzechenzhang/GMP}.}

\noindent\textbf{Baseline Methods.} 
We compare against the single-level graph neural preconditioner (GNP)~\cite{chen2025graph}, as well as the classical preconditioners, Jacobi, ILUT~\cite{saad1994ilut}, and AMG. 
The ILUT preconditioner is
provided by
\path{scipy.sparse.linalg.spilu} 
using default parameters. The AMG method is obtained 
from PyAMG~\cite{pyamg2023}, specifically using the \path{pyamg.blackbox.solver()} routine. The configuration of the blackbox solver is computed via \path{pyamg.blackbox.solver_configuration}. We also experimented with the Approximate Ideal Restriction (AIR) method~\cite{manteuffel2018nonsymmetric} (\path{pyamg.classical.air_solver}) but found it to be less robust than \path{blackbox} (consistent with the findings of~\citet{chen2025graph}). Therefore, we adopt the \path{blackbox} solver for the primary comparisons, while retaining AIR as a baseline for the case study in Section~\ref{sec:case_study} 
and in the supplementary results in Appendix \ref{sec:air_results}.
The GNP model is trained according to the exact hyperparameter configuration specified by~\citet{chen2025graph}. 

We pre-scale each test matrix $A$ using an upper bound on its spectral radius derived from the Gershgorin estimate:
\begin{equation*}
        \hat{A} = A/\gamma \quad \text{where} \quad \gamma = \min \left\{ \max_{i} \left\{ \sum_{j} |a_{ij}| \right\}, \max_{j} \left\{ \sum_{i} |a_{ij}| \right\} \right\}.
\end{equation*} 
We then apply the parameter-free scaling $s(\cdot) = \sqrt{n}/\|b\|_2$, and scale back by $s^{-1}$ to ensure the model remains scale-equivariant. 

All the preconditioners are used
with the flexible GMRES (FGMRES) method.
For all linear systems, the exact solution 
is assumed to be $x^* = \mathbf{1}$, and the
initial guess is set to $x_0 =\mathbf{0}$. 
The stopping criteria for FGMRES are 
a reduction of the relative residual
norm to \texttt{rtol}$=10^{-8}$ or reaching a maximum number of iterations \texttt{imax}=$100$.

\noindent\textbf{Compute Environment.} All code is implemented in Python with PyTorch and PyTorch-Geometric. All experiments are conducted on a machine equipped with a single NVIDIA Tesla A100 GPU with 40GB of memory.

\subsubsection{GMP configurations}
\label{sec:gmp_setup}
To support evaluation across more than 800 matrices, we adopt the following practical design choices.

\noindent\textbf{(1) Node-only message passing (i.e., no edge feature updates).}
We omit all edge update steps in the MPNN formulation (Eq.~\eqref{eq:mpnn_edge}) and train the pipeline end-to-end using \emph{only} node features.
This choice is due to computational considerations: applying an MLP to every edge can be a dominant cost when a matrix contains up to $\sim\!2$M nonzeros, making the large-scale evaluation impractical.

\noindent\textbf{(2) AMG initialization with robustness checks.}
We generate the initial AMG hierarchy using PyAMG by applying the sequence of:
\texttt{classical\_\allowbreak strength\_\allowbreak of\_\allowbreak connection}
\(\rightarrow\)
\texttt{pyamg.\allowbreak classical.\allowbreak split.\allowbreak RS} 
(the Ruge--St\"uben coarsening algorithm)
\(\rightarrow\)
\texttt{pyamg.\allowbreak class\allowbreak ical.\allowbreak interpolate.\allowbreak classical\_\allowbreak interpolation}.
The resulting interpolation operator $P$ 
is used to 
initialize the node features associated with 
the learned operators  
\(\mathcal{P}_\theta\) as discussed in Sec.~\ref{sec:bipartite_att}.
To ensure numerical stability, we check for invalid entries; 
if detected, they are replaced with 
random values, and each row is renormalized to sum to one. We initialize the learning-based restriction
operator
\(\mathcal{R}_\theta\) 
with \(R=P^\top\),
and allow the bipartite cross-attention module to learn distinct restriction and interpolation operators
\(\mathcal{R}_\theta\) and \(\mathcal{P}_\theta\).

\noindent\textbf{(3) Training protocol.}
For each training instance, we sample a target solution \(x\) and form the corresponding right-hand side \(b = Ax\). The GMP model is trained to output a correction \(e=\mathrm{GMP}_{\theta}(b)\), where at training time the input is the right-hand side \(b\) in place of the inference-time Krylov residual \(r\). We then minimize the \(\ell_1\)-norm of the residual
\begin{equation}
\label{eq:l1_loss}
\mathcal{L}(\theta)=\|A e-A x\|_1 \;=\; \|A e - b\|_1,
\end{equation}
which is a commonly used robust regression objective. The target solutions \(x\) are drawn from a mixture of the standard normal distribution \(\mathcal{N}(\mathbf{0}, \mathbf{I}_n)\) and the distribution \(\mathcal{N}(\mathbf{0}, \mathbf{\Sigma}_m^{\mathbf{x}})\), where the latter uses the Arnoldi process to capture the important eigen-subspace of \(A\) corresponding to the smallest eigenvalues.
We optimize with Adam~\cite{adam} for \(2000\) epochs using a learning rate of \(10^{-3}\).

\noindent\textbf{(4) GMP levels for training and testing.}
We focus primarily on two-level GMP in the experimental section 
and demonstrate its extension to the multilevel setting
on a subset of matrices; see Section~\ref{sec:multi-single} for details. 
There are two main reasons for emphasizing the two-level GMP. 
First, classical AMG employs multiple levels to avoid solving a large system
at the coarsest grid. On the other hand, GMP replaces the coarse-grid solve 
with a (multilayer) ResGCN, which significantly reduces the need for additional coarsening levels and makes a two-level hierarchy effective in practice.
Second, for the non-SPD problems considered, the standard convergence theory of AMG does not apply, and increasing levels can even degrade convergence.
The optimal hierarchy is highly problem-dependent, making it challenging to tune across 867 matrices in our test suite.

\noindent\textbf{(5) GMP hyperparameters.}
We do not tune hyperparameters. The 867 matrices constitute a collection of
general sparse linear systems that cannot be assumed to share a common distribution.
Consequently, one cannot assume a shared structure that a single model can exploit. We therefore adopt a \emph{per-matrix} training protocol rather than learning a shared model. 
Since tuning hyperparameters separately for each matrix is computationally prohibitive, we instead use
a fixed configuration across the entire dataset. 
When such a distributional assumption \emph{does} hold---e.g., when matrices arise from the same PDE family with varying parameters---cross-matrix training becomes not only feasible but beneficial.
In such settings,
the learned operators generalize to unseen problems and even unseen grid resolutions. 
We demonstrate this regime on a controlled Poisson/anisotropic-diffusion benchmark in Appendix~\ref{sec:cross_matrix}.

For the MPNNs in both the smoother and coarse-grid 
solve modules, we use the same ResGCN backbone as in GNP~\cite{chen2025graph}, stacking 8 layers with input/output dimension 16.
We use 2-layer MLPs with hidden dimension 32 for feature lifting and projection.
We do not include a post-smoothing stage.
For the bipartite cross-attention module, we use 4 attention heads with hidden dimension 32. We do not apply dropout in any layer.

\subsection{Metrics}
\label{sec:metrics}

We evaluate the performance of preconditioners along two axes: \emph{robustness} and \emph{convergence}.

\paragraph{Robustness}
We report a failure rate accounting for two disjoint failure modes.
\emph{Construction failures} occur when the preconditioner cannot be built, e.g., due to numerical breakdown.
\emph{Solution failures} occur when the Krylov iteration becomes unreliable, as indicated by a significant inconsistency between internally monitored residual norms and explicitly computed residual norms.

\paragraph{Convergence}
Restricted to successful runs, we quantify convergence using two complementary metrics.

\noindent\textbf{(1) Iter-AUC.}
Following prior work~\cite{chen2025graph}, 
we report the area under the relative residual curve (on a logarithmic scale) with respect to iteration count:
\begin{equation}
\label{eq:iter_auc}
\mathrm{Iter\mbox{-}AUC}
\;=\;
\sum_{i=0}^{\mathtt{iters}}\Big(\log_{10} r_k - \log_{10}(\mathtt{rtol})\Big),
\quad
r_k \;=\; \frac{\|b-Ax_k\|_2}{\|b\|_2},
\end{equation}
where $\mathtt{iters}$ denotes the number of iterations actually executed when FGMRES terminates (capped by the iteration budget).
Iter-AUC summarizes the \emph{entire} convergence trajectory 
and remains informative for difficult cases 
that make substantial progress 
but do not reach \(\texttt{rtol}\) within the iteration budget.

\noindent\textbf{(2) Iteration count to tolerance.}
Because Iter-AUC is an aggregate trajectory-level metric, 
we additionally report the smallest number of iterations that reaches the target tolerance:
\begin{equation}
\label{eq:kstar}
k^\star
\;=\;
\min \left\{ k : \|r_k\|_2 < \mathtt{rtol} \; \mbox{and} \; k < \mathtt{imax} 
\right\}.
\end{equation}

Iter-AUC captures \emph{global convergence behavior}, 
but it does not explicitly distinguish between (i) whether the solver 
reaches the prescribed tolerance and 
(ii) how quickly it does so.
In particular, distinct convergence profiles can achieve similar 
Iter-AUC values; for example, a method that rapidly reduces the residual but plateaus above \texttt{rtol} may have a comparable Iter-AUC to one that converges more slowly but successfully. For this reason, we report failure rate to assess robustness, use \(k^\star\) to quantify efficiency on successful solves, and employ Iter-AUC 
as a stable trajectory-level summary for both solved and near-solved cases.

\subsection{Main results}

We now address the questions posed at the beginning of this section. 

\subsubsection{Multilevel vs. Single-level}\label{sec:multi-single}
We begin with a comparison between the proposed GMP preconditioner, evaluated in its two-level configuration, and the single-level graph neural preconditioner (GNP)~\cite{chen2025graph}. Table~\ref{tab:gmp-vs-gnp} reports a head-to-head win-rate comparison on the subset of matrices where \emph{both} methods converge, so that the results isolate the convergence behavior rather than the robustness.
GMP achieves a lower Iter-AUC on \(62.1\%\) of the matrices, indicating stronger trajectory-level residual reduction on average.
Moreover, GMP increasingly dominates as the convergence tolerance becomes stricter: 
the win rate increases from \(48.1\%\) at \(\texttt{rtol}=10^{-2}\) to \(63.5\%\) at \(\texttt{rtol}=10^{-5}\), whereas GNP 
achieves a win rate of \(21.0\%\)\footnote{Note: The percentage in Table~\ref{tab:gmp-vs-gnp} and Table~\ref{tab:three_way_winrate} may not sum to 100\% due to rounding.} and \(17.2\%\), respectively.
These results suggest that incorporating a coarse-grid correction yields consistent advantages beyond single-level message passing, particularly when high-accuracy solutions are required. 

\begin{table}[t]
\centering
\caption{Head-to-head win rates (\%) of GMP vs.\ GNP under different convergence metrics.}
\label{tab:gmp-vs-gnp}
\begin{tabular}{l rrrrr}
\toprule
         & Iter-AUC & $10^{-2}$ & $10^{-3}$ & $10^{-4}$ & $10^{-5}$ \\
\midrule
GNP wins & 37.9\%   & 21.0\%    & 23.8\%    & 19.7\%    & 17.2\%    \\
Tie      & ---      & 30.8\%    & 22.8\%    & 19.7\%    & 19.3\%    \\
GMP wins & 62.1\%   & 48.1\%    & 53.4\%    & 60.7\%    & 63.5\%    \\
\bottomrule
\end{tabular}
\end{table}

The comparison above fixes GMP to two levels. We now ask whether deeper hierarchies can help. Selecting the best number of levels requires a separate training for each matrix, which is prohibitively expensive for all the 867 matrices; we therefore conduct this multilevel study on the first 100 non-SPD SuiteSparse matrices (in alphabetical order from the dataset), training GMP with 2 to 5 levels. For each matrix, we select the best number of levels according to the final residual norm. The results are in Table~\ref{tab:multilevel_winrate}.

\begin{table}[t]
\centering
\caption{Win rate of GMP vs GNP by final residual norm (2-level only vs. best level across 2--5 levels).}
\label{tab:multilevel_winrate}
\small
\setlength{\tabcolsep}{5pt}
\begin{tabular}{lcc}
\toprule
\textbf{Tests} & \textbf{2-level} & \textbf{GMP best levels ($\geq$2)} \\
\midrule
All matrices       & 65.6\% & 72.0\% \\
Matrices that AMG reaches $10^{-2}$   & 70.5\% & 77.3\% \\
Matrices that AMG reaches $10^{-4}$   & 74.1\% & 77.8\% \\
Matrices that AMG reaches $10^{-5}$   & 76.0\% & 80.0\% \\
\bottomrule
\end{tabular}
\end{table}

For context, \texttt{pyamg.blackbox} (the reference traditional AMG solver) uses an average of only 2.6 levels on these 100 matrices (53\% use 2, 35\% use 3, 11\% use 4, 1\% use 5), and 2.7 levels on all 867 non-SPD matrices---classical AMG itself rarely finds deep hierarchies beneficial on non-SPD problems. The best GMP level distribution follows the same trend: 2-level is optimal on 71\%, 3-level on 12\%, 4-level on 9\%, and 5-level on 9\%, with an average best depth of 2.5---closely matching AMG's average.

We use AMG convergence as a proxy for hierarchy quality: when AMG converges to a tight tolerance, the underlying hierarchy is reliable, and GMP can effectively exploit it at deeper levels. We draw two conclusions from the results in Table~\ref{tab:multilevel_winrate}:

\noindent(1) Multilevel improves over 2-level. Tuning for the best number of
    levels consistently provides ${\sim}4\%$--$7\%$ improvement in win rate over the 2-level approach (e.g., $65.6\% \rightarrow 72.0\%$ for all the matrices).

\noindent(2) Better hierarchy quality improves GMP. Comparing the rows of Table~\ref{tab:multilevel_winrate}, GMP's win rate increases as the evaluation is restricted to matrices on which AMG can converge to progressively tighter tolerances: from 72.0\% on all matrices to 80.0\% on those where AMG reaches a residual tolerance $10^{-5}$. 
This trend suggests that the primary limitation of deeper GMP on non-SPD matrices is the quality of the underlying multigrid hierarchy rather than instability in the learned components. When the hierarchy itself is ineffective, adding more levels cannot compensate for the deficiency and may in fact degrade performance due to the increased instability of the hierarchy.

\begin{table*}[!t]
\centering
\caption{Failures of preconditioners (count and percentage).}
\label{tab:failures}
\small
\setlength{\tabcolsep}{10pt}
\begin{tabular}{lccccc}
\toprule
 & \textbf{GMP} & \textbf{GNP} & \textbf{ILUT} & \textbf{AMG} & \textbf{Jacobi} \\
\midrule
Construction failure & 54 (6.2\%) & 0 (0.00\%) & 348 (40.14\%) & 62 (7.15\%) & N/A \\
Solution failure     & 2 (0.23\%) & 1 (0.12\%) & 61 (7.04\%)   & 5 (0.58\%)  & 53 (6.11\%) \\
\bottomrule
\end{tabular}
\end{table*}

\begin{table*}[t!]
\centering
\caption{Wall-clock timing breakdown on the 867 non-SPD matrices. Setup: ILUT factorization, AMG hierarchy construction, or GNP/GMP training.
Solve: GMRES solve phase (all iterations).}
\label{tab:timing-867}
\begin{tabular}{lrrrr}
\toprule
Method & Setup Med. (s) & Setup Avg. (s) & Solve Med. (s) & Solve Avg. (s) \\
\midrule
No Precond & ---    & ---    & 0.54 & 0.54  \\
ILUT & 0.04   & 93.79  & 0.26  & 11.26  \\
AMG & 0.18   & 103.22 & 1.31 & 50.87 \\
GNP & 37.18  & 55.25  & 2.20 & 2.20  \\
GMP & 135.77 & 210.90 & 8.08 & 9.17  \\
\bottomrule
\end{tabular}
\end{table*}

\subsubsection{Robustness.} We summarize the failure rates in Table~\ref{tab:failures}. While GMP and AMG exhibit comparable overall failure rates ($\approx$6--7\%), their failure mechanisms differ fundamentally. AMG failures are primarily numerical, errors indicating the presence of invalid values (e.g., infs or NaNs), despite the safeguards implemented in \path{pyamg.blackbox} to handle zero or near-zero diagonal entries. In contrast, all GMP construction failures (6.2\%) are attributed exclusively to out-of-memory (OOM) constraints rather than algorithmic instability. We perform a detailed breakdown for the OOM occurrence in Appendix \ref{sec:OOM}. Turning to algorithmic stability, this distinction is further highlighted by ILUT, which fails on approximately 40\% of the test problems due to numerical singularities. Finally, GNP exhibits the highest stability with no construction failures and only a single solution failure.

\subsubsection{Timing.} We provide a full timing breakdown on the 867 non-SPD matrices in Tab.~\ref{tab:timing-867}. The reported timing reflects a single fixed hyperparameter configuration (i.e., $B{=}16$, $H{=}4$, $d{=}32$, $L{=}3$, training for 2000 epochs) applied uniformly across all matrices. This configuration is not tuned per matrix and may be larger than necessary for some problems; for a specific application, a smaller model or batch size may suffice and can reduce both training and per-iteration costs. As an illustration, we report a single-matrix memory--accuracy--runtime trade-off study in Appendix \ref{sec:time_reduce}.

\subsubsection{Convergence Efficiency.}
Table~\ref{tab:three_way_winrate} reports three-way win rates on the matrices where all methods produce valid runs.

Iter-AUC measures trajectory-level progress, while the columns labeled \(10^{-2}\) -- \(10^{-5}\)
measure the convergence speed \(k^\star\).

\noindent\textit{GMP vs.\ GNP.} 
Across all three comparison groups (Jacobi, ILUT, and AMG),
GMP consistently achieves higher win rates than GNP under both Iter-AUC and strict tolerance metrics. 
Notably, in the AMG comparison group, 
GMP outperforms GNP on a substantial 
fraction of matrices (e.g., \(27.2\%\) at \(\texttt{rtol}=10^{-5}\)), whereas GNP wins only \(3.7\%\) under the same criterion. This gap indicates that a multilevel coarse-grid correction enables convergence improvements that are rarely 
attained by the single-level GNP baseline.

\noindent\textit{Against classical baselines.} Table~\ref{tab:three_way_winrate} should be interpreted as follows: the win rates are computed \emph{only} on matrices where all compared methods yield valid runs, and
therefore
reflect \emph{convergence efficiency conditional on success}. 
Under this conditioning, (i) {Jacobi} is consistently dominated by GMP (e.g., \(55.4\%\) Iter-AUC wins and \(71.4\%\) wins at \(\texttt{rtol}=10^{-5}\)), highlighting 
the substantial gains that learned preconditioning provides over simple relaxation schemes. (ii) {ILUT}  overwhelmingly outperforms
on the subset where it succeeds, 
which is expected when the factorization is
numerically
stable; however, as shown in Table~\ref{tab:failures}, 
ILUT frequently fails to construct. Thus, 
its strong performance applies only in the conditional regime where it is numerically feasible. (iii) AMG remains a strong general-purpose baseline, yet GMP is competitive, 
achieving a \(22\%\)–\(32\%\) win share across metrics and \(27.2\%\) wins at \(\texttt{rtol}=10^{-5}\). This indicates that learned multilevel correction can accelerate convergence on a substantial subset of systems where the classical AMG method is less effective.

\begin{table}[t]
\centering
\caption{\textbf{Three-way win rate (\%) among valid runs.}
The numbers are the percentage of matrices on which the method attains the best score (ties broken uniformly).}
\label{tab:three_way_winrate}

\begin{tabular}{ll ccccc}
\toprule
& Method & Iter-AUC & $10^{-2}$ & $10^{-3}$ & $10^{-4}$ & $10^{-5}$ \\
\midrule
\multirow{3}{*}{\shortstack{vs\\Jacobi}}
& Jacobi & 10.6\% & 11.3\% & 11.4\% & 17.9\% & 23.4\% \\
& GNP    & 34.0\% & 12.6\% & 12.5\% &  8.4\% &  5.2\% \\
& GMP    & \textbf{55.4\%} & \textbf{76.2\%} & \textbf{76.1\%} & \textbf{73.7\%} & \textbf{71.4\%} \\
\midrule
\multirow{3}{*}{\shortstack{vs\\ILUT}}
& ILUT    & \textbf{77.3\%} & \textbf{95.3\%} & \textbf{95.3\%} & \textbf{97.7\%} & \textbf{97.3\%} \\
& GNP    &  8.6\% &  0.8\% &  1.4\% &  0.6\% &  0.7\% \\
& GMP    & 14.1\% &  3.9\% &  3.3\% &  1.7\% &  2.0\% \\
\midrule
\multirow{3}{*}{\shortstack{vs\\AMG}}
& AMG    & \textbf{46.1\%} & \textbf{69.5\%} & \textbf{73.6\%} & \textbf{71.1\%} & \textbf{69.1\%} \\
& GNP    & 21.6\% &  6.0\% &  4.0\% &  5.2\% &  3.7\% \\
& GMP    & 32.3\% & 24.6\% & 22.4\% & 23.7\% & 27.2\% \\
\bottomrule
\end{tabular}
\end{table}


\begin{figure*}[!t] 
    \centering 
    \includegraphics[width=0.95\linewidth]{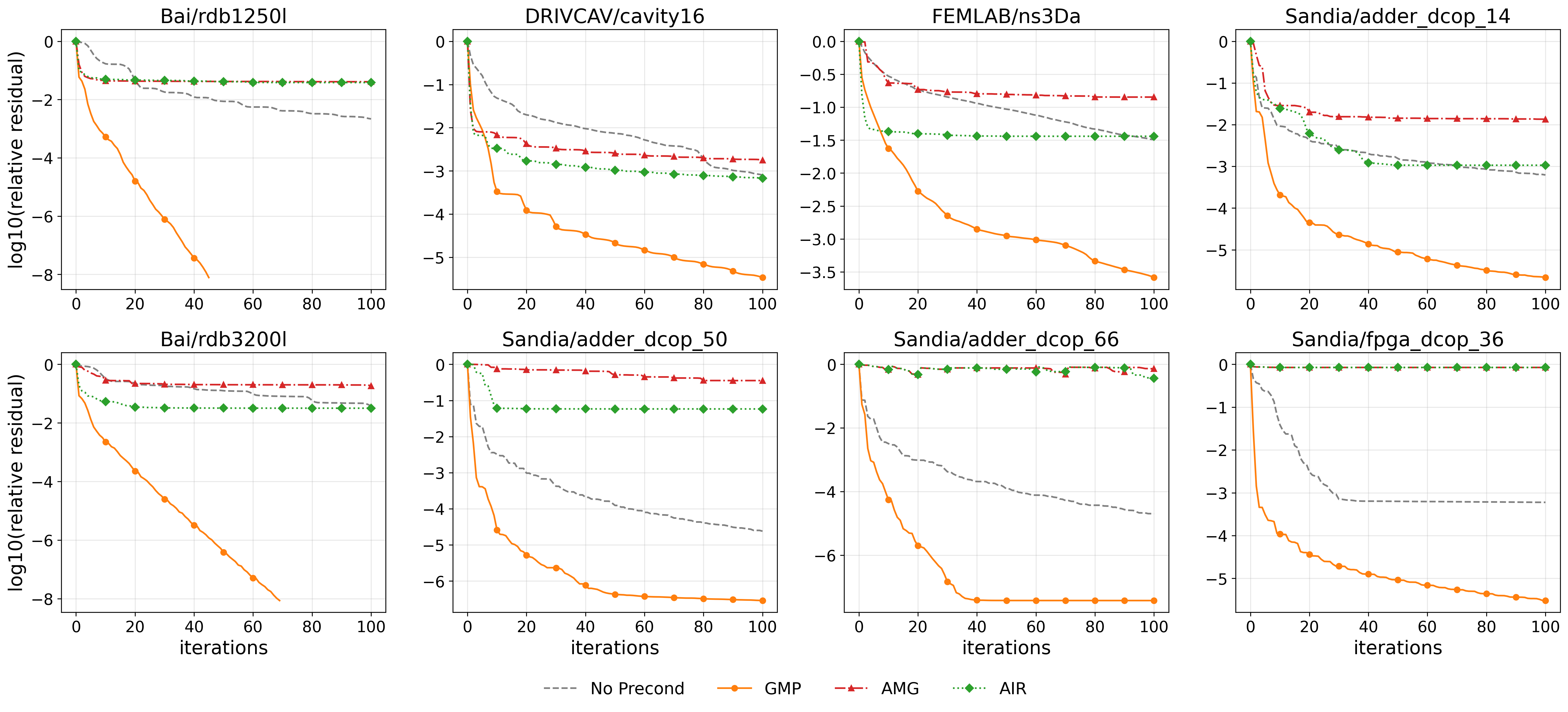} \caption{Convergence trajectories on representative challenging nonsymmetric systems. Log$_{10}$ relative residual versus FGMRES iteration.  
    We compare no preconditioner, AMG, AIR, and the proposed GMP.}
    \Description{A multi-panel set of line plots showing convergence trajectories on representative challenging nonsymmetric systems. Each panel plots the base-ten logarithm of the relative residual on the vertical axis against FGMRES iteration count on the horizontal axis, comparing four methods: no preconditioner, AMG, AIR, and the proposed GMP. The GMP curves generally show sustained residual reduction, often reaching lower residual levels than AMG and AIR.}
 \label{fig:case_study} 
\end{figure*}

\subsection{Case Study on Challenging Matrices}
\label{sec:case_study}

In this section, we analyze representative scenarios where GMP outperforms classical AMG to identify problem domains in which learning-based preconditioning 
offers a distinct advantage. 
Among these instances, we identify 28 different categories, the most prevalent of which arise from circuit simulation, optimization problems, and computational fluid dynamics (CFD). 
These matrices exhibit extreme numerical variability, with condition numbers ranging from $2.71$ to $1.05 \times 10^{65}$, and matrix norms spanning from $2.44 \times 10^{-10}$ to $1.00 \times 10^{20}$.

Notably, in 8.9\% of these cases, GMP achieves a \emph{strong win}, defined as successfully converging to a tolerance of $10^{-6}$ while the classical AMG method stagnates at an early stage.
To further investigate this regime, we selected eight representative matrices for detailed analysis. 
For the experiments with these matrices, we additionally 
include the AIR method obtained from \path{pyamg.classical.air_solver} as a baseline, since it is a specialized nonsymmetric AMG method.
As shown in Fig.~\ref{fig:case_study}, FGMRES preconditioned with classical AMG and AIR fail to achieve substantial residual reduction, performing no 
better than non-preconditioned FGMRES. Conversely,
the GMP preconditioner yields significantly faster convergence and attains a 
remarkably greater reduction in the residual norm.
This behavior indicates that the learning-based multilevel approach can be more 
effective than classical AMG in the nonsymmetric regime. 
In such problems, fundamental assumptions underlying classical AMG are often 
violated. In contrast, GMP does not rely on fixed algebraic heuristics; instead, 
it learns directly from data, allowing it to adapt to highly nonsymmetric cases. 
Moreover, the multilevel structure of GMP combines local message-passing-based 
relaxation with learned coarse-grid correction, 
enabling it to capture global error components that are poorly represented by 
classical AMG heuristics. Together, these factors help explain why GMP is able to achieve robust residual reduction in cases where both classical AMG and AIR struggle.

For a balanced comparison, we also include in Appendix \ref{sec:case_study_amg_app} several examples in which AMG and AIR outperform GMP.

\section{Conclusion}
In this work, we present \emph{Graph Neural Multilevel Preconditioner} (GMP), a residual-conditioned learned preconditioner that lifts classical AMG into a data-driven setting by (i) reusing an AMG-generated hierarchy as a \emph{structural prior} and (ii) learning multilevel operators end-to-end. A key ingredient is parameterizing restriction and interpolation with \emph{bipartite cross-attention} and allowing \emph{separate} parameterizations for the two transfer directions, enabling Petrov--Galerkin-style coarse operators and making the framework naturally applicable to non-SPD and nonsymmetric systems.

\paragraph{Future work.}
Our current implementation incurs, on average, approximately \(4\times\) higher per-application cost than the single-level GNP due to multilevel traversal and bipartite attention. Exploring \textit{cross-matrix generalization} for graph-neural preconditioners on general sparse systems is therefore a particularly meaningful direction: if a single model could generalize across matrices (or across matrix families), the training cost would be amortized over many systems rather than paid per matrix, turning learned preconditioning from a per-instance procedure into a reusable component.

\section{Limitations and Ethical Considerations}
This work studies learned preconditioners and evaluates them on publicly available numerical benchmarks. Consequently, common concerns such as demographic bias and privacy leakage are not directly applicable. The primary impact of this work is computational: improving solver efficiency can reduce runtime and energy consumption in scientific and engineering workloads.

\section{Use of GenAI}
We employed Claude 4.5 Opus to assist with code implementation, and ChatGPT-5.2 to assist with paper revision. 

\begin{acks}
    This work was performed under the auspices of the U.S. Department of Energy by Lawrence Livermore National Laboratory under Contract DE-AC52-07NA27344 (LLNL-PROC-2020426) and was supported by the LLNL-LDRD program under Project No. 24-ERD-033.
\end{acks}

\bibliographystyle{ACM-Reference-Format}
\balance
\bibliography{reference_numerical, reference_ml}

\appendix

\section{Appendix}
\subsection{Theorem}
\label{sec:theorem_a1}
\begin{theorem}[Stability of Attention-Based Interpolation]
Let $P \in \mathbb{R}^{n \times n_c}$ be the interpolation operator with the
interpolation weights given by
\begin{equation}
P_{ij} = \alpha_{ij}, \quad i \in \mathcal{V}_F, \; j \in C_i \ ,    
\end{equation}
where $\alpha_{ij} \ge 0$ and $\sum_{j \in C_i} \alpha_{ij} = 1$, as produced by a normalized (softmax) attention mechanism. 
For coarse points $i \in \mathcal{V}_C$, let $P_{ii} = 1$ and $P_{ij} = 0$ for $j \neq i$.
Then the following properties hold:
\begin{enumerate}
    \item  $\|P\|_{\infty} = 1$;
    \item  For any coarse vector $v_c$, $\|Pv_c\|_{\infty} \le \|v_c\|_{\infty}$;
    \item  For any constant coarse-grid vector $v_c = c\mathbf{1}_c$, the interpolated vector satisfies $Pv_c = c\mathbf{1}_f$.
\end{enumerate}
\end{theorem}

\begin{proof}
For any point $i \in \mathcal{V}_F$ and coarse-grid vector $v_c$, we have
\[
|(Pv_c)_i| = \left| \sum_{j \in C_i} \alpha_{ij} (v_{c})_j \right|
\le \sum_{j \in C_i} \alpha_{ij} |(v_{c})_j|
\le \|v_c\|_{\infty} \sum_{j \in C_i} \alpha_{ij}
= \|v_c\|_{\infty}.
\]
For coarse points $i \in \mathcal{V}_C$, $(Pv_c)_i = (v_{c})_i$ by definition. Hence,
$\|Pv_c\|_{\infty} \le \|v_c\|_{\infty}$,
which implies $\|P\|_{\infty} \le 1$. Equality follows since $P$ contains identity rows on coarse points.

For constant vectors $v_c = c\mathbf{1}_c$, we have
\[
(Pv_c)_i = \sum_{j \in C_i} \alpha_{ij} c = c
\quad \text{for all } i \in \mathcal{V}_F,
\]
and trivially $(Pv_c)_i = c$ for $i \in \mathcal{V}_C$. This proves constant reproduction.
\end{proof}
In energy minimization approaches \cite{10.1007/978-3-540-34469-8_2, doi:10.1137/100803031} for AMG interpolation,
the energy-stability is typically controlled by minimizing
$\operatorname{Tr}(P^\top A P)$, or equivalently minimizing each column $p$ 
of $P$ in the $A$-norm.
For diffusion-type and $M$-matrix operators, this 
admits the expression
\begin{equation}
p^\top A p \sim \sum_{i<j} |a_{ij}| (p_i - p_j)^2,
\end{equation}
which promotes coarse basis functions that 
vary slowly across strongly connected graph edges $|a_{ij}|$. 
However, this 
criterion alone does not control the magnitude of the interpolation 
weights and permits large but slowly varying values, which can lead 
to unstable coarse-to-fine transfer. The row-stochastic constraint 
enforced by the attention-based construction provides a complementary 
stability mechanism  in the $\ell_{\infty}$ norm.

\subsection{Results on PyAMG AIR}
\label{sec:air_results}

Compared with classical \texttt{pyamg.\allowbreak aggregation.\allowbreak smoothed\allowbreak\_aggregation\allowbreak\_solver} (SA), the AIR method~\cite{manteuffel2018nonsymmetric} is a preferable choice for non-SPD matrices. However, due to long construction time on specific matrices (e.g., from groups GHS\_indef, Schenk\_IBMNA, Rajat), which can take hours to process a single instance, we report results for only the first 647 out of 867 matrices here. 

\begin{table}[H]
\centering
\caption{Head-to-head win rates (\%) of GMP vs.\ AIR under different convergence metrics.}
\label{tab:gmp-vs-air}
\begin{tabular}{ll rrrr}
\toprule
         & Method & Iter-AUC & $10^{-2}$ & $10^{-3}$ & $10^{-4}$ \\
\midrule
vs       & AIR    & 38.6\%   & 70.8\%    & 79.2\%    & 86.1\%    \\
AIR      & GNP    & 25.6\%   & 4.5\%     & 5.7\%     & 1.4\%     \\
         & GMP    & 35.9\%   & 24.7\%    & 15.1\%    & 12.5\%    \\
\bottomrule
\end{tabular}
\end{table}

\subsection{GMP Failure Analysis}
\label{sec:OOM}

We note that the experiments are conducted on a shared GPU cluster, where dedicated use of a 40GB node is not guaranteed. When we re-ran all the cases that had failed due to OOM on a dedicated 40GB A100, the failure rate dropped significantly, as shown in Tab.~\ref{tab:new_failures}. The two construction failures are one from the classical Ruge–Stüben coarsening selecting no coarse points, yielding an empty interpolation operator $P$, and one from OOM. We point out that GMP achieves the same optimal linear asymptotic scaling with respect to the interpolation operator, $O(\mathrm{nnz}(P))$, as classical AMG; the increased memory cost reflects the large constant factors from the network architecture. Specifically, GMP's memory cost is dominated by the cross-attention mechanism, scaling as $O(\mathrm{nnz}(P) \times B \times H \times d \times L)$, alongside the ResGCN smoother cost of $O(n \times B \times e \times L_{\mathrm{GCN}})$ (where $B=16$ is the batch size, $H=4$ is the number of attention heads, $d=32$ is the attention dimension, $L=3$ is the number of attention layers, $e=16$ is the GCN embedding dimension, and $L_{\mathrm{GCN}}=8$ is the number of GCN layers). While the asymptotic $O(\mathrm{nnz}(P))$ complexity matches classical AMG's scalar operations, GMP carries a much larger constant factor of $B \times H \times d \times L = 6{,}144$ due to batched multi-head attention. We note that $B$, $H$, $d$, and $L$ are treated as fixed hyperparameters and used uniformly across all 867 matrices without per-problem tuning. For a given problem, a smaller configuration may suffice. An illustrative example can be found in Appendix \ref{sec:time_reduce}.

\begin{table}[H]
\centering
\caption{Failures of preconditioners (count and percentage).}
\label{tab:new_failures}
\small
\setlength{\tabcolsep}{10pt}
\begin{tabular}{lcc}
\toprule
 & \textbf{GMP} & \textbf{GNP} \\
\midrule
Construction failure & 2 (0.2\%) & 0 (0.00\%) \\
Solution failure     & 2 (0.23\%) & 1 (0.12\%)  \\
\bottomrule
\end{tabular}
\end{table}

\subsection{Cross-Matrix Generalization}
\label{sec:cross_matrix}

The per-matrix training protocol used in our main experiments (Sec.~\ref{sec:gmp_setup}) is motivated by the fact that the selected  matrices from SuiteSparse
constitute a collection of diverse sparse linear systems
that cannot be assumed to share a common distribution. When such an
assumption \emph{does} hold---e.g., when matrices arise from the same PDE problem with varying parameters---cross-matrix training is not only feasible but beneficial.

To demonstrate this, we generated 100 anisotropic diffusion matrices on a $32\times32$ grid (varying the anisotropy ratio and rotation angle), split $80/20$ for train/test, and trained a \emph{single} GMP model on the 80 training matrices. At test time, we evaluate on 20 unseen matrices at each of three grid sizes---$32\times32$, $48\times48$, and $64\times64$---by building the AMG hierarchy for each matrix and applying the shared learned weights without any further training. Results are reported in Tab.~\ref{tab:cross_matrix}.

\begin{table}[H]
\centering
\caption{Generalization on anisotropic diffusion matrices (median GMRES iterations, $\mathtt{rtol}=10^{-8}$). Per-matrix methods are trained from scratch at each grid size. Cross-matrix GMP is trained on 80 matrices ($32\times32$) and applied to all sizes without retraining.}
\label{tab:cross_matrix}
\small
\setlength{\tabcolsep}{6pt}
\begin{tabular}{llccc}
\toprule
\textbf{Method} & \textbf{Training} & \textbf{32$\times$32} & \textbf{48$\times$48} & \textbf{64$\times$64} \\
\midrule
GNP & per-matrix   & 74 & 133 & 205 \\
GMP & per-matrix   & 40 & 66  & 97  \\
\textbf{GMP} & \textbf{cross-matrix} & \textbf{24} & \textbf{44} & \textbf{60} \\
\bottomrule
\end{tabular}
\end{table}

The cross-matrix GMP \emph{outperforms} the per-matrix baselines at all grid sizes---even on $48\times48$ and $64\times64$, where the baselines were trained at that size. This demonstrates two points: (1) cross-matrix training produces a more robust preconditioner; and (2) the learned operators generalize to unseen grid sizes, enabling a ``train small, apply large'' deployment.

\subsection{Memory--Accuracy--Runtime Trade-off}
\label{sec:time_reduce}

The default GMP configuration used throughout the main experiments is fixed and applied uniformly to all matrices, and is therefore not tuned to any individual problem. To examine how far the configuration can be reduced on a single problem, we conduct a hyperparameter study on \texttt{Bai/rdb1250l}, a 
non-SPD matrix 
randomly selected from the dataset,
where $n{=}1{,}250$ and $\mathrm{nnz}{=}7{,}300$, 
to explore the memory--accuracy--runtime trade-off. We vary the batch size $B$, the number of attention heads $H$, the attention dimension $d$, the number of attention layers $L$, and the number of epochs for each configuration. Peak GPU memory is measured via \texttt{torch.cuda.max\_memory\_allocated()}. 
The results are reported in Tab.~\ref{tab:time_reduce}.

\begin{table*}[t]
\centering
\caption{Memory--accuracy--runtime trade-off on \texttt{Bai/rdb1250l} ($n{=}1{,}250$, $\mathrm{nnz}{=}7{,}300$). ``Iters'' and ``Res'' are the GMRES iterations and final relative residual; ``Peak Mem'' is the peak GPU memory.}
\label{tab:time_reduce}
\small
\setlength{\tabcolsep}{4pt}
\begin{tabular}{lccccccl}
\toprule
\textbf{Config} & \textbf{Epochs} & \textbf{Iters} & \textbf{Res} & \textbf{Train} & \textbf{Solve} & \textbf{Peak Mem} & \textbf{vs Default} \\
\midrule
Default ($B{=}16, H{=}4, d{=}32, L{=}3$) & 2{,}000 & 300 & 1.60e-2 & 87.7s & 7.6s & 304M & --- \\
$B{=}1, H{=}4, d{=}32, L{=}3$            & 4{,}000 & \textbf{91}  & 9.02e-9 & 161.3s & 2.3s & 36M & 8.4$\times$ less mem \\
$B{=}1, H{=}1, d{=}16, L{=}1$            & 2{,}000 & \textbf{166} & 9.53e-9 & 58.9s  & 3.0s & 25M & 12.2$\times$ less mem, 1.5$\times$ faster train \\
\bottomrule
\end{tabular}
\end{table*}

\begin{figure*}[!t]
    \centering
    \includegraphics[width=0.95\linewidth]{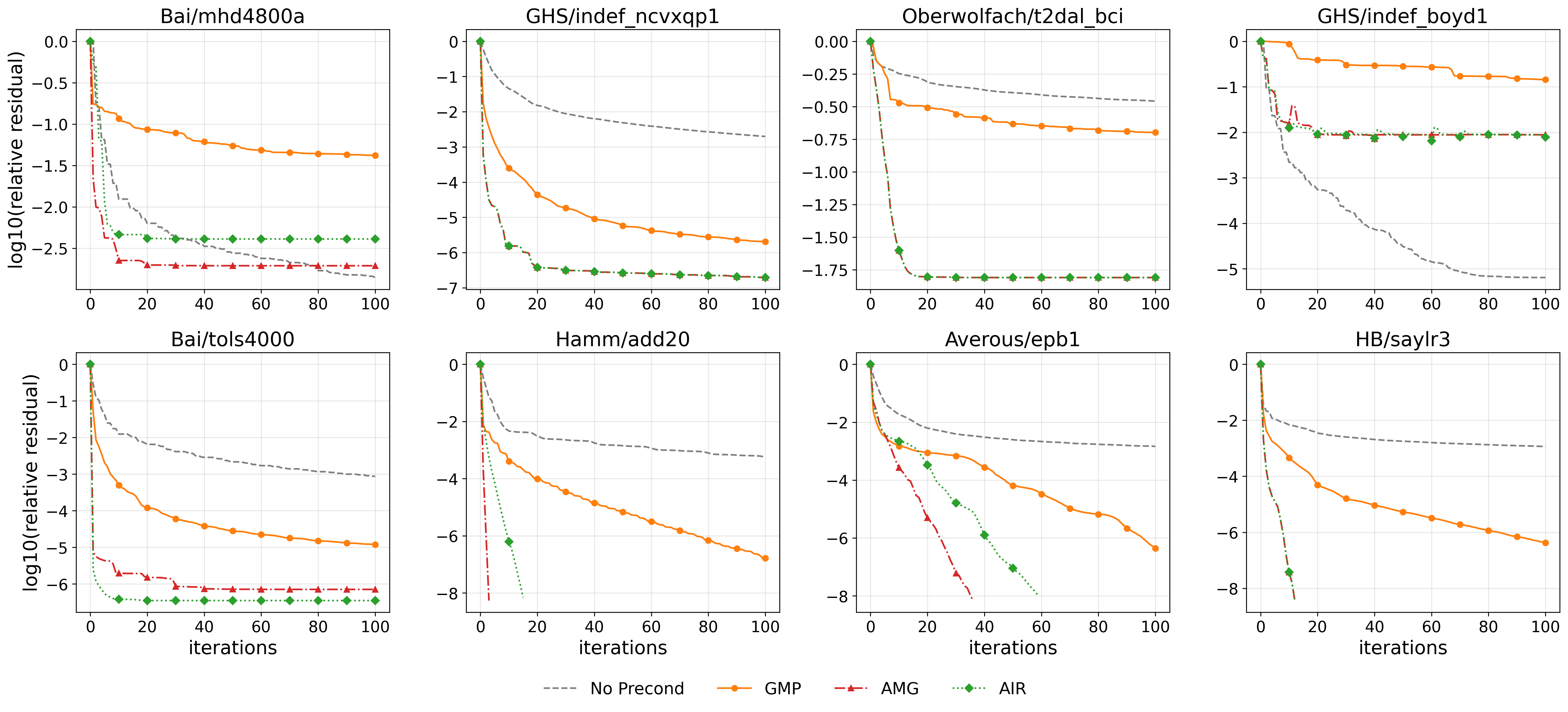}
    \caption{Convergence trajectories on representative nonsymmetric systems where the classical preconditioners outperform GMP. Log$_{10}$ relative residual versus FGMRES iteration, comparing no preconditioner, AMG, AIR, and the proposed GMP.}
    \Description{A multi-panel set of line plots showing convergence trajectories on representative nonsymmetric systems where classical preconditioners outperform GMP. Each panel plots the base-ten logarithm of the relative residual on the vertical axis against FGMRES iteration count on the horizontal axis, comparing four methods: no preconditioner, AMG, AIR, and the proposed GMP. In these cases the AMG and AIR curves reach lower residual levels than GMP.}
    \label{fig:case_study_amg}
\end{figure*}

Key observations are as follows:
\begin{enumerate}
    \item Reducing $B$ from 16 to 1 while keeping the model architecture ($H{=}4$, $d{=}32$, $L{=}3$) yields an $8.4\times$ memory reduction and the best convergence on this matrix (91 iterations to a relative residual of $10^{-8}$), whereas the default configuration does not reach this tolerance within the iteration budget.
    \item An aggressive reduction ($B{=}1$, $H{=}1$, $d{=}16$, $L{=}1$) achieves a $12.2\times$ memory reduction, converges in 166 iterations, and trains $1.5\times$ faster than the default.
    \item For this matrix, both reduced configurations produce better preconditioners than the default, suggesting that the default hyperparameters are larger than necessary here. We emphasize that this is a single-matrix study: it does not establish that such reductions are uniformly beneficial or free of accuracy cost, and the best configuration is problem-dependent. Nonetheless, peak memory is dominated by the cross-attention term, so reducing these hyperparameters lowers memory usage; and since this term scales with $\mathrm{nnz}(P)$, we expect the savings to grow on larger matrices.
\end{enumerate}

\subsection{Case Study Additional Results}
\label{sec:case_study_amg_app}
In Sec.~\ref{sec:case_study}, we presented cases in which GMP outperforms the classical preconditioners. For a balanced comparison, we additionally report cases in which the classical preconditioners AMG and AIR outperform GMP, shown in Fig.~\ref{fig:case_study_amg}.

\end{document}